\documentclass[10pt]{amsart}
\usepackage{amsmath}
\usepackage{amssymb,amscd}
\usepackage{graphicx}
\usepackage{mathdots}
\usepackage{mathrsfs} 
\usepackage{graphicx}
\usepackage{bbm} 
\usepackage{microtype} 

\usepackage{colonequals} 
\usepackage{mathtools}

\usepackage{tikz,tkz-euclide}
\usepackage{tikz-cd}
\usetikzlibrary{arrows}
\usepackage{tikz-3dplot}

\usepackage{color}\definecolor{darkblue}{rgb}{0,0.1,.5}
\usepackage[colorlinks=true,linkcolor=darkblue, urlcolor=darkblue, citecolor=darkblue]{hyperref}
\usepackage{cleveref}

\theoremstyle{plain}
\newtheorem{theorem}{Theorem}[section]
\newtheorem{lemma}[theorem]{Lemma}
\newtheorem{proposition}[theorem]{Proposition}
\newtheorem{corollary}[theorem]{Corollary}

\theoremstyle{definition}
\newtheorem{definition}[theorem]{Definition}
\newtheorem{example}[theorem]{Example}

\newtheorem{conjecture}[theorem]{Conjecture}
\newtheorem{construction}[theorem]{Construction}

\theoremstyle{remark}
\newtheorem*{remark}{Remark}

\numberwithin{equation}{section}

\def \begineq{\begin{equation}}
\def \endeq{\end{equation}}

\def \bb{\mathbb}

\def \N{{\bb{N}}}
\def \Z{{\bb{Z}}}
\def \Q{{\bb{Q}}}
\def \R{{\bb{R}}}

\def \Tor{\mathop{\mathrm{Tor}}\nolimits}

\def\Bier{\mathrm{Bier}}
\def\pol{\mathrm{pol}}
\def\M{\mathrm{max}}
\def\MF{\mathrm{min}}

\newcommand{\vc}{\mathrm{vc}}
\newcommand{\sk}{\mathrm{sk}}
\newcommand{\del}{\mathrm{del}}
\newcommand{\link}{\mathrm{link}}

\def\zk{\mathcal Z_K}

\DeclareMathAlphabet{\mathbbmsl}{U}{bbm}{m}{sl}

\makeatletter
\@namedef{subjclassname@2020}{%
  \textup{2020} Mathematics Subject Classification}
\makeatother

\title[Simplicial vs. cubical spheres, polyhedral products and the Nevo-Petersen conjecture]{Simplicial vs. cubical spheres, polyhedral products\\ and the Nevo-Petersen conjecture}

\author[Limonchenko]{Ivan Limonchenko}
\address{Mathematical Institute of the Serbian Academy of Sciences
and Arts (SASA), Belgrade, Serbia}
\email{ivan.limoncenko@turing.mi.sanu.ac.rs}

\author[\v{Z}ivaljevi\'c]{Rade \v{Z}ivaljevi\'c}
\address{Mathematical Institute of the Serbian Academy of Sciences
and Arts (SASA), Belgrade, Serbia}
\email{rade@turing.mi.sanu.ac.rs}



\begin{document}

\maketitle

\begin{abstract}
We prove that a Murai sphere is flag if and only if it is a nerve complex of a flag nestohedron and classify all the polytopes arising in this way. Our classification implies that flag Murai spheres satisfy the Nevo-Petersen conjecture on $\gamma$-vectors of flag homology spheres. We continue by showing that a Bier sphere is minimally non-Golod if and only if it is a nerve complex of a truncation polytope different from a simplex and classify all the polytopes arising in this way. Finally, the notion of a cubical Bier sphere is introduced based on the polyhedral product construction, and we study combinatorial and geometrical properties of these cubical complexes.
\end{abstract}

\section{Introduction}

This paper is devoted to the combinatorial, algebraic and topological aspects related to Bier spheres and their generalizations. Since 1992, when a beautiful and simple construction of a PL-sphere being the deleted join of a simplicial complex different from a simplex and its Alexander dual was introduced in~\cite{Bier}, the class of these Bier spheres was studied  intensively, see~\cite{Matousek} and several generalizations of the original construction were proposed. Among those relevant to our paper, we mention here the generalizations obtained in the framework of topological combinatorics~\cite{JNPZ} (Bier complexes), poset theory~\cite{BPSZ05} (Bier posets) and combinatorial commutative algebra~\cite{Mu} (generalized Bier spheres, or Murai spheres).  

It turned out that Bier spheres and their generalizations carry a rich combinatorial structure and enjoy a few important properties like being shellable and edge decomposable, making them a useful tool in topological combinatorics~\cite{Matousek},  polytope theory~\cite{Zivaljevic21}, optimization theory~\cite{Zivaljevic19}, game theory~\cite{TZJ} and combinatorial commutative algebra~\cite{HK, Mu}. At the same time, the ideas, constructions and results from all those areas of research enriched the theory of Bier spheres. Indeed, methods of geometric combinatorics and combinatorial commutative algebra yield new proofs of the fact that Bier's construction produces PL-spheres, see~\cite{dL} and Gorenstein complexes, see~\cite{Mu}, methods of homological algebra allowed to compute certain bigraded Betti numbers of Bier spheres, see~\cite{HK}, and methods of topological combinatorics provided a construction of a canonical complete simplicial fan for each Bier sphere therefore showing that all Bier spheres are starshaped, see~\cite{Zivaljevic19}. 

Recently, in~\cite{LS} the theory of Bier spheres was put in the context of toric topology~\cite{BP02, TT}. In particular, it was shown that all Bier spheres have maximal possible Buchstaber numbers, canonical quasitoric manifolds and small covers over all 2-dimensional Bier spheres were constructed. On the other hand, the classification type problems were in the spotlight of~\cite{LS}: all the 13 different combinatorial types of two-dimensional Bier spheres were identified and, since the ``Simplicial Steiniz Problem'' has a positive solution in dimension two, all the 13 corresponding 3-dimensional convex simple polytopes were also described. It turned out that all those polytopes admit Delzant realizations~\cite[Construction 2.1.3]{TT}, all but one are nestohedra~\cite{FS, P05} and some of them are generalized truncation polytopes~\cite{L14}.    

In this paper we establish new links between the theory of Bier spheres and geometric combinatorics, combinatorial commutative algebra, and toric topology. Two types of problems are considered here. 

The first type of a question we are concerned with is the classification problem. In~\cite{HK} a classification of all flag Bier spheres was obtained. Recall that a simplicial complex is called \emph{flag} if its minimal non-faces have cardinality not greater than two. In Section 3, we provide a different proof of this result and  
show that all flag Bier spheres $\Bier(K)$ are polytopal and the corresponding simple polytopes $P_K$ are flag nestohedra. 

Recall that an $n$-dimensional \emph{nestohedron} is a convex simple polytope equal to the Minkowski sum of simplices, generated by standard basis vectors in $\R^{n+1}$ that are indexed by elements of a \emph{building set} on $[n+1]:=\{1,2,\ldots,n+1\}$, see~\cite[Definition 1.5.10]{TT}. The construction of a nestohedron corresponding to a connected building set as a result of cutting faces of a simplex can be found in~\cite[Theorem 1.5.18]{TT}. 

We show that the following classification statement holds.

\begin{theorem}
A Bier sphere $\Bier(K)$ is flag if and only if it is polytopal with $P_K$ being a flag nestohedron of one of the following types:
\begin{enumerate}
    \item $P_K=I^n$ for $n\geq 1$;
    \item $P_K=I^n\times P_5$ for $n\geq 0$;
    \item $P_K=I^n\times P_6$ for $n\geq 0$;
    \item $P_K=I^n\times Q_2^3$ for $n\geq 0$. 
\end{enumerate}
\end{theorem}

In the above theorem, $P_m$ for $m\geq 3$ denotes the $m$-gon and $Q_2^3$ denotes the 3-dimensional simple polytope, which can be obtained from the 3-dimensional cube $I^3$ by two adjacent edge cuts.
 
The theory of \emph{2-truncated cubes}, i.e. simple polytopes combinatorially equivalent to the ones obtained from a cube by cutting off its codimension two faces by hyperplanes in general position was developed in~\cite{BV} and applied to toric topology in~\cite{L16, BL}. This theory shows that a nestohedron is flag if and only if it is a 2-truncated cube. E.g. a building set for $Q_2^3$ can be found in~\cite[Theorem 2.16]{LS}. 
The theory of 2-truncated cubes together with the main result of~\cite{AV} show that the following Nevo-Petersen conjecture holds for all flag Bier spheres. 

\begin{conjecture}[\cite{NP}]
If $K$ is a flag homology sphere, then its $\gamma$-vector $\gamma(K)$ is the $f$-vector of a flag simplicial complex.    
\end{conjecture}

In Section 4 we generalize the above classification result to Murai spheres. It turns out that all flag Murai spheres $\Bier_c(M)$ are also polytopal and the corresponding simple polytopes $P_{M,c}$ are the same flag nestohedra as in the previous theorem. Namely, the next statement holds. 

\begin{theorem}
A Murai sphere $\Bier_c(M)$ is flag if and only if it is polytopal with $P_{M,c}$ being a flag nestohedron of one of the following types:
\begin{enumerate}
    \item $P_{M, c}=I^n$ for $n\geq 1$;
    \item $P_{M, c}=I^n\times P_5$ for $n\geq 0$;
    \item $P_{M, c}=I^n\times P_6$ for $n\geq 0$;
    \item $P_{M, c}=I^n\times Q_2^3$ for $n\geq 0$.
\end{enumerate}
\end{theorem}

It follows immediately that the Nevo-Petersen conjecture holds for all generalized Bier spheres.

\begin{corollary}
The Nevo-Petersen conjecture holds in the class of flag Murai spheres.
\end{corollary}

In Section 5 we continue working on the classification problem and discuss \emph{Koszul homology} (cohomology algebra of the Koszul d.g.a.) of face rings $\Bbbk[\Bier(K)]$ of Bier spheres $\Bier(K)$. These studies were initiated in~\cite{HK}, where certain \emph{bigraded Betti numbers} of Bier spheres, i.e. ranks of bigraded components of the corresponding Koszul homology algebras were computed. Koszul homology of a face ring known in toric topology as the (multigraded) \emph{Tor-algebra} $\Tor_{\Bbbk[m]}(\Bbbk[K],\Bbbk)$ of a simplicial complex $K$ over a commutative ring with a unit $\Bbbk$ became one of the key objects of study in the theory of polyhedral products~\cite{BBCG} due to the fundamental result~\cite[Theorem 4.5.4]{TT}. It establishes an isomorphism of $\Bbbk$-algebras between the cohomology algebra $H^*(\zk;\Bbbk)$ of a moment-angle-complex $\zk$ and the Koszul homology of the corresponding face ring $\Bbbk[K]$.

We are interested in the following two important algebraic properties of a face ring and therefore topological properties of the corresponding moment-angle-complex. The first notion dates back to~\cite{G}, where a class of local rings with rational Poincar\'e series was identified. Let $K$ be a simplicial complex with $m$ vertices. Its face ring $\Bbbk[K]$ is called \emph{Golod} if multiplication and all higher Massey products vanish in $\Tor^{+}_{\Bbbk[m]}(\Bbbk[K],
\Bbbk)$. The second notion was introduced in~\cite{BJ} in the framework of combinatorial commutative algebra. The ring $\Bbbk[K]$ is called \emph{minimally non-Golod} if it is not Golod and the ring $\Bbbk[\del_K(i)]$ is Golod for each $i\in [m]$. Here $\del_K(i)$ is the \emph{deletion} of $i$ from $K$ i.e. the full subcomplex in $K$ on $[m]\setminus\{i\}$. We say that $K$ is a \emph{Golod complex} (\emph{minimally non-Golod complex}) if its face ring $\Bbbk[K]$ is Golod (minimally non-Golod) for any field $\Bbbk$, respectively. 

Combinatorial classification of Golod and minimally non-Golod complexes is far from being complete even in low dimensions. However, some criteria characterizing Golodness and minimal non-Golodness in terms of related algebraic and topological properties were obtained in~\cite{GPTW, GIPS, LP}. We prove the following classification result for Bier spheres.

\begin{theorem}
Let $K\neq\Delta_{[m]}$ be a simplicial complex on $[m]$ with $m\geq 3$. The following statements hold.
\begin{itemize}
\item[(a)] $\Bier(K)$ is Golod if and only if either $K=\partial\Delta_{[m]}$ or $K^\vee=\partial\Delta_{[m']}$. In this case $\Bier(K)$ is polytopal and we have 
$$
P_K=\Delta^{m-1};
$$
\item[(b)] $\Bier(K)$ is minimally non-Golod if and only if either $K$ or $K^\vee$ equals the set of $\ell$ disjoint points, where $1\leq\ell\leq m$. In this case $\Bier(K)$ is polytopal and $P_K$ is a truncation polytope obtained from $\Delta^{m-1}$ by cutting off $\ell$ of its vertices; that is
$$
P_K=\vc^{\ell}(\Delta^{m-1}).
$$
\end{itemize}
\end{theorem}

The second type of a question we are concerned with is the problem of finding close relatives (natural  generalizations) of the Bier sphere construction. Here we make use of the polyhedral product construction as follows. Polyhedral products $(X, A)^K = \mathcal{Z}_K(X, A)$ are among the central objects (constructions) of toric topology~\cite{BP02, BBCG, TT}. In Section~\ref{sec:poly-prod} we review them in their relation to Alexander duality. In Section~\ref{sec:meet-1} we show that there exists a canonical isomorphism 
\[
  B(K,K^\vee)_{cubic} \cong \partial Z(K,K^\vee)
\]
where  $B(K,K^\vee)_{cubic}$ is the natural cubulation of the Bier sphere $B(K,K^\vee) = \Bier(K)$ described in Section \ref{sec:canonical_cubulation} and $\partial Z(K,K^\vee)$ is the boundary of the intersection 
\[
Z(K,K^\vee) := \mathcal{Z}_K(J, J_{\leqslant 0})\cap \mathcal{Z}_{K^\vee}(J, J_{\geqslant 0})  
\]
of two fundamental polyhedral products, associated to $K, K^\vee$ and $J=[-1,1]$.

This construction provides a natural place for Bier spheres in the theory of polyhedral products. It also sheds new light on some basic (and somewhat neglected) aspects of geometry and combinatorics of monomial ideals such as the ``theory of corners'' of D. Bayer, see \cite{B-preprint} and \cite[p. 86]{Sturm2005}.

We introduce key definitions and constructions used in this paper in Section 2 and further, where we need them. For more details on simplicial complexes, face rings and polyhedral products, we refer the reader to the monographs \cite{Matousek}, \cite{St}, and \cite{TT}, respectively.  

\section{Basic definitions and constructions}

In this section we collect those definitions, constructions and results concerning simplicial complexes that are crucial throughout the paper. 

In what follows we fix a positive integer $m\in\N$ and set $[m]:=\{1,2,\ldots,m\}$.

\begin{definition}
We say that a subset $K\subseteq 2^{[m]}$ is an (abstract) {\emph{simplicial complex}} on $[m]$ if the following two conditions hold:
\begin{enumerate}
\item $\varnothing\in K$; and
\item $\sigma\in K, \tau\subseteq\sigma\Rightarrow \tau\in K$.
\end{enumerate}
Elements of $K$ are called its {\emph{faces}} or {\emph{simplices}}. Faces of cardinality $1$ are called {\emph{vertices}}, faces of cardinality $2$ are called {\emph{edges}}. The set of vertices of $K$ is denoted by $V(K)$. The {\emph{dimension}} of $K$ is defined to be one less than the maximal cardinality of a face of $K$ and is denoted by $\dim(K)$. 
\end{definition}

It is easy to see that a simplicial complex $K$ on $[m]$ is determined by its set $\max(K)$ of maximal (w.r.t. inclusion) simplices called its {\emph{facets}}, as well as by the set $\MF(K)$ of minimal (w.r.t. inclusion) subsets in $[m]$ that are not elements of $K$ called its {\emph{minimal non-faces}}.
Minimal non-faces of cardinality $1$ are called {\emph{ghost vertices}}. We also say that maximal faces of a simplicial complex \emph{generate} this complex and write
$$
K=\langle I\subseteq [m]\,|\,I\in\max(K)\rangle.
$$

For any simplicial complex $K$ on $[m]$ and a subset $I\subseteq [m]$ we can define two important subcomplexes in $K$, namely
\begin{itemize}
\item the \emph{full subcomplex} in $K$ on $I$: 
$$
K_I := \{J\subseteq I\,|\,J\in K\} = K\cap 2^I;
$$ 
\item the \emph{link} of $I$:
$$
\link_K(I) := \{J\subseteq I^c\,|\,I\sqcup J\in K\},
$$
where $I^c:=[m]\setminus I$. \end{itemize}

\begin{example}
Observe that, by definition, for any simplicial complex $K$ on $[m]$ one has $\link_K(\varnothing)=K$, $I\in\max(K)$ if and only if $K_I=\Delta_{I}:=2^I$, the simplex with the vertex set $I$, and $I\in\MF(K)$ if and only if $K_I=\partial\Delta_I:=2^I\setminus\{I\}$, the boundary of the simplex with the vertex set $I$. The smallest simplicial complex on $[m]$ is the void complex $\varnothing_{[m]}:=\{\varnothing\}$.   
\end{example}

An algebraic object which is closely related to a simplicial complex $K$ and which carries exactly the same amount of information about $K$ as the \emph{combinatorial type} of $K$ (i.e. the face lattice isomorphism class) is the face ring of $K$. In what follows $\Bbbk$ denotes $\Z$ or a field and we use the notation 
$$
\Bbbk[m]:=\Bbbk[v_1,\ldots,v_m]\text{ and } v_I:=v_{i_1}\ldots v_{i_p}\in \Bbbk[m],
$$
where $I=\{i_1,\ldots,i_p\}\subseteq [m]$. If $I=\varnothing$, then $v_I:=1$.

\begin{definition}
Let $K$ be a simplicial complex on $[m]$. Its \emph{Stanley-Reisner ring} (or, \emph{face ring}) over $\Bbbk$ is a monomial ring
$$
\Bbbk[K]:=\Bbbk[m]/I_K,
$$
where the \emph{Stanley-Reisner ideal} (or \emph{face ideal}) of $K$ is a monomial ideal
$$
I_K:=(v_I\,|\,I\in\min(K)).
$$
\end{definition}

\begin{example}
Here are the first examples.
\begin{itemize}
\item If $K=\varnothing_{[m]}$, then $I_K=(v_1,\ldots,v_m)$ and $\Bbbk[K]=\Bbbk$;
\item If $K=\Delta_{[m]}$, then $I_K=(0)$ and $\Bbbk[K]=\Bbbk[m]$;
\item If $K=\partial\Delta_{[m]}$, then $\Bbbk[K]=\Bbbk[m]/(v_1\ldots v_m)$;
\item If $K$ is the 4-cycle on $[4]$ labeled clockwise, then 
$$
\Bbbk[K]=\Bbbk[v_1,v_2,v_3,v_4]/(v_1v_3, v_2v_4);
$$
\item If $K$ is the 5-cycle on $[5]$ labeled clockwise, then 
$$
\Bbbk[K]=\Bbbk[v_1,v_2,v_3,v_4,v_5]/(v_1v_3, v_1v_4, v_2v_4, v_2v_5, v_3v_5).
$$
\end{itemize}
\end{example}

\begin{definition}
A simplicial complex $K$ is called {\emph{flag}} if one of the following equivalent conditions holds:
\begin{itemize}
\item elements of $\MF(K)$ have cardinalities $\leq 2$;
\item each set of vertices of $K$ pairwise linked by edges in $K$ is itself in $K$.
\end{itemize} 
A simplicial complex $K$ is called \emph{pure} if all maximal faces of $K$ have the same cardinality. 
\end{definition}

In what follows we usually do not distinguish between an abstract simplicial complex on $[m]$ and its geometric realization in the Euclidean space, as well as between $i\in [m]$ and $\{i\}\in 2^{[m]}$. Among examples of pure simplicial complexes of key importance for our studies we should name the classes of combinatorial spheres. Here are those classes we shall deal with in this paper.

\begin{definition}
We call an $(n-1)$-dimensional simplicial complex $K$
\begin{itemize}
\item a \emph{polytopal sphere}, if it is isomorphic to the boundary of a simplicial $n$-polytope $Q$. In this case we write $K=K_P$ for the simple polytope $P$ dual to $Q$ and say that $K$ is a \emph{nerve complex} of $P$;
\item a \emph{starshaped sphere}, if there is a geometric realization of $K$ in $\R^n$ and a point $p$ in $\R^{n}$ such that each ray emanating from $p$ meets this realization in exactly one point;
\item \emph{PL-sphere}, if it is PL-homeomorphic to $\partial\Delta^{n}$;
\item \emph{simplicial sphere}, if it is homeomorphic to $S^{n-1}$;
\item \emph{(rational) homology sphere}, if for each $\sigma\in K$ one has:
\[
\tilde{H}_{i}(\link_{K}(\sigma);\Q)=\begin{cases}
0,&\text{if $i<n-1-|\sigma|$;}\\
\Q,&\text{if $i=n-1-|\sigma|$.}
\end{cases}
\]
\end{itemize} 
\end{definition}

Observe that the five classes of combinatorial spheres listed above form an increasing sequence w.r.t. inclusion when viewed from top to bottom. Moreover, when $n=2$ these classes coincide with each other. 

\begin{remark}
It was shown in~\cite{St} that the class of (rational) homology spheres coincides with the class of (rational) Gorenstein* complexes studied in combinatorial commutative algebra. 
\end{remark}

Let $[m]:=\{1,2,\ldots,m\}$ and $[m']:=\{1',2',\ldots,m'\}$ be two ordered sets with the map $\phi\colon i\mapsto i', 1\leq i\leq m$ being an order preserving bijection between them. Denote by $I'=\phi(I)$ the image of a subset $I\subseteq [m]$. 

Here is the first class of spheres we are going to study in this paper, it was introduced and proved to be contained in the class of PL-spheres in~\cite{Bier}. An explicit construction of the required PL-homeomorphism was given in~\cite{dL}. 

\begin{definition}
Suppose $K$ is a simplicial complex on $[m]$ and $K\neq\Delta_{[m]}$. Then its {\emph{Alexander dual}} is a simplicial complex $K^\vee$ on $[m']$ such that 
$$
I\in \MF(K)\Longleftrightarrow [m']\setminus I'\in \M(K^\vee).
$$
By the {\emph{Bier sphere}} over $K$ we mean a simplicial complex $\Bier(K)$ on $[m]\sqcup [m']$ such that 
$$
\Bier(K)=\{I\sqcup J'\,|\,I\in K, J'\in K^\vee, I\cap J=\varnothing\}.
$$
That is, the Bier sphere over $K$ is defined to be the {\emph{deleted join}} of $K$ and $K^\vee$.
\end{definition}

It is easy to classify all one-dimensional Bier spheres. 

\begin{example}\label{BierOneDimExample}
Direct application of the above definition shows that a 1-dimensional Bier sphere $\Bier(K)$ is combinatorially equivalent to the boundary of 
\begin{itemize}
\item a triangle, if $K=\langle\{1,2\}, \{1,3\}, \{2,3\}\rangle$;
\item a square, if $K=\langle\{1,2\}, \{2,3\}\rangle$ or $K=\langle\{1,2\}\rangle$;
\item a pentagon, if $K=\langle\{1\}, \{2,3\}\rangle$; 
\item a hexagon, if $K=\langle\{1\}, \{2\}, \{3\}\rangle$.
\end{itemize} 
\end{example}

Asymptotically most of the Bier spheres are non-polytopal, see~\cite{Matousek}, although no particular example of a non-polytopal Bier sphere has been constructed so far, see~\cite[Problem 4.7]{LS}. Furthermore, Bier spheres were proven to be starshaped spheres in~\cite{Zivaljevic19}.

The second class of spheres we are going to study in this paper is the class of generalized Bier spheres, or the Murai spheres introduced in~\cite{Mu}. Up to the end of this section, we fix a positive integer vector $c=(c_1,\ldots,c_m)\in\N^m$ and we set $\bar{c}=(c_{1}+1,\ldots,c_{m}+1)\in\N^{m}$. In our discussion of Murai spheres we use the notations from~\cite{Mu}.

The Murai sphere construction arose in combinatorial commutative algebra, it is based on the notion of a multicomplex, which generalizes that of a simplicial complex and it makes use of the Alexander duality theory for monomial ideals. 

The multicomplex-theoretic Alexander duality goes as follows. 

\begin{definition}
By a $c$-\emph{multicomplex} we mean a non-empty set $M$ of $c$-\emph{monomials} in $\Bbbk[m]$, that is
$$
x^{a}:=x_1^{a_1}\cdot\ldots\cdot x_{m}^{a_m}\in M\text{ with } a_i\leq c_i\text{ for all }1\leq i\leq m
$$
such that if $m_1\in M$ and $m_2$ divides $m_1$, then $m_2\in M$.

Its {\emph{Alexander dual}} w.r.t. $c$ is a $c$-multicomplex $M^\vee$ defined by
$$
M^\vee := \{(x^{a})^c\,|\,x^{a}\text{ is a }c-\text{monomial in }\Bbbk[m]\text{ such that }x^{a}\notin M\},
$$
where
$$
(x^{a})^c:=x_1^{c_1-a_1}\cdot\ldots\cdot x_m^{c_m-a_m}.
$$
We call a $c$-multicomplex $M$ $c$-{\emph{full}} if it is the set of all $c$-monomials; otherwise, $M$ is a {\emph{proper $c$-multicomplex}}. 
\end{definition}

Note that for $c=(1,\ldots,1)$ a $c$-multicomplex is essentially the same object as a simplicial complex on $[m]$. Similarly to this particular case, in general we define the sets $\max(M)$ and $\min(M)$ for any $c$-multicomplex $M$. We also say that maximal elements of $M$ \emph{generate} this multicomplex and write
$$
M=\langle p\in\Bbbk[m]\,|\,p\in\max(M)\rangle.
$$

The ideal-theoretic Alexander duality goes as follows. 

\begin{definition}
Let $I\subset \Bbbk[m]$ be a $c$-\emph{ideal}, that is, $I$ is generated by $c$-monomials. Its {\emph{Alexander dual w.r.t. $c$}} is a $c$-ideal $I^\vee$ in $\Bbbk[m]$ defined by
$$
I^\vee := \{(x^{a})^c\,|\,x^{a}\text{ is a }c-\text{monomial in }\Bbbk[m]\text{ such that }x^{a}\notin I\}.
$$ 
\end{definition}

Denote by $I_c(M)\subset S$ the monomial ideal generated by all $c$-monomials not in $M$. Note that for $c=(1,\ldots,1)$ this is the Stanley-Reisner ideal of the simplicial complex corresponding to $M$. It is easy to see that
$$
(I_c(M))^\vee = I_c(M^\vee).
$$

Let us describe the set on which we are going to define a Murai sphere. 

Given $i, 1\leq i\leq m$ define 
$$
\tilde{X}_i := \{x^{(0)}_i,\ldots,x^{(c_i)}_i\}.
$$
and set 
$$
\tilde{X} := \tilde{X}_1\cup\tilde{X}_2\cup\ldots\cup\tilde{X}_m.
$$
Finally, for each $c$-monomial $x^a$, let
$$
F_c(x^a) := \tilde{X}\setminus \{x^{(a_1)}_1,\ldots,x^{(a_m)}_m\}.
$$

Now, we are ready to give a definition of a Murai sphere. 

\begin{definition}
For a proper $c$-multicomplex $M$, we define its {\emph{generalized Bier sphere}}, or the {\emph{Murai sphere}} to be a simplicial complex $\Bier_c(M)$ on $\tilde{X}$ generated by the following set of its maximal faces:
$$
\M(\Bier_c(M)) := \{F_c(x^a)\setminus\{x^{(j)}_i\}\,|\,x^a\in M, x^{a}x^{-a_{i}+j}_i\notin M, a_i < j \leq c_i\}.
$$
\end{definition}

In~\cite[Proposition 1.10]{Mu} it was proved that $\Bier_c(M)$ is a $(|c|-2)$-dimensional simplicial sphere. Furthermore, \cite[Theorem 1.10]{Mu} shows that for $c=(1,\ldots,1)$ the above definition yields a classical Bier sphere of the corresponding simplicial complex.

\begin{remark}
For any proper $c$-multicomplex $M$, its generalized Bier sphere $\Bier_c(M)$ is shellable (\cite[Theorem 2.1]{Mu}) and edge decomposable (\cite[Theorem 4.6]{Mu})  
\end{remark}

Finally, let us describe the set $\min(\Bier_c(M))$ of minimal non-faces of a Murai sphere $\Bier_c(M)$ in terms of its Stanley-Reisner ideal.
First, we define two polarizations for a $\bar{c}$-ideal in $\Bbbk[m]$; they will both belong to the polynomial algebra on the set of variables 
$$
X := X_1\cup X_2\cup\ldots\cup X_m,\text{ where }X_i := \{x_{i,0},\ldots,x_{i,c_i}\}\text{ for }1\leq i\leq m.
$$.

Given a monomial ideal $I$, we denote by $G(I)$ the unique minimal set of monomial generators of $I$.

\begin{definition}
The \emph{polarization} of a $\bar{c}$-ideal $I\subset\Bbbk[m]$ is the monomial ideal $\pol(I)$ in $\Bbbk[X]$ 
such that
$$
\pol(I):=(\pol(x^a):=\prod\limits_{a_i\neq 0}x_{i,0}\ldots x_{i,a_{i}-1}\,|\,x^a\in G(I))\subset\Bbbk[X].
$$

The $\ast$-\emph{polarization} of a $\bar{c}$-ideal $I\subset\Bbbk[m]$ is the monomial ideal $\pol^\ast(I)$ in $\Bbbk[X]$ such that
$$
\pol^\ast(I):=(\pol^\ast(x^a):=\prod\limits_{a_i\neq 0}x_{i,c_i}\ldots x_{i,c_i-a_{i}+1}\,|\,x^a\in G(I))\subset\Bbbk[X].
$$
\end{definition}

The next statement proved in~\cite[Theorem 3.6]{Mu} describes the face ideal of $\Bier_c(M)$ and hence its set of minimal non-faces $\min(\Bier_c(M))$.

\begin{theorem}[\cite{Mu}]\label{FaceIdealMuraiThm}
Let $M$ be a proper $c$-multicomplex. Then we have:
$$
I_{\Bier_c(M)}=\pol(I_c(M))+\pol^\ast(I_c(M^\vee))+\pol(x_1^{c_{1}+1},\ldots,x_{m}^{c_{m}+1}).
$$
\end{theorem}

\begin{remark}
A purely algebraic proof of the fact that $\Bier_c(M)$ is a Gorenstein complex (over any field), i.e. a join of a simplex and a homology sphere, based on the above result was given in~\cite[Theorem 5.3]{Mu}.  
\end{remark}

When $c=(1,\ldots,1)$, Theorem~\ref{FaceIdealMuraiThm} follows from the well-known description of the set of minimal non-faces $\min(\Bier(K))$ of a Bier sphere $\Bier(K)$, see also~\cite[Proposition 2.6]{LS}. On the other hand, it is important to note that the union of ideal generators in the r.h.s. may not be in one-to-one correspondence with the set of minimal non-faces even in the Bier sphere case, see~\cite[Example 3.8]{Mu}. 

\section{Bier spheres and face vectors}

In this section we discuss some combinatorial and geometric properties of Bier spheres as well as the Nevo-Petersen conjecture for flag homology spheres. We show that for flag Bier spheres that conjecture is true.

Let us recall the definition of face vectors of a simplicial complex.

\begin{definition}
Let $K$ be an $(n-1)$-dimensional simplicial complex. Its \emph{$f$-vector} is the tuple $(f_{-1}, f_0,\ldots,f_{n-1})$, where $f_i$ is the number of $i$-dimensional faces in $K$. Alongside with the $f$-vector, the following face vectors of $K$ are considered:
\begin{itemize}
\item \emph{$h$-vector} $h(K)=(h_0,h_1,\ldots,h_n)$:
$$
h_0t^n+\ldots+h_{n-1}t+h_n=(t-1)^n+f_0(t-1)^{n-1}+\ldots+f_{n-1};
$$
\item \emph{$\gamma$-vector} $\gamma(K)=(\gamma_0,\gamma_1,\ldots,\gamma_{[n/2]})$:
$$
h_0+h_1t+\ldots+h_nt^n=\sum\limits_{i=0}^{[n/2]}\gamma_it^i(1+t)^{n-2i}.
$$
\end{itemize}
The last representation exists provided the $h$-vector of $K$ is symmetric, so this is the class of simplicial complexes for which we consider their $\gamma$-vectors.    
\end{definition}

Since homology spheres satisfy the Dehn-Sommerville relations, see~\cite{St}, their $h$-vectors are symmetric and hence one can consider $\gamma$-vectors of homology spheres. 

Our main goal in this section is to show that in the class of flag Bier spheres the following conjecture holds.

\begin{conjecture}[\cite{NP}]
If $K$ is a flag homology sphere, then its $\gamma$-vector $\gamma(K)$ is the $f$-vector of a flag simplicial complex.    
\end{conjecture}

Let us discuss the combinatorial types of flag Bier spheres. The next lemma contains a few easy observations that are useful in the classification of all flag Bier spheres.

\begin{lemma}\label{BierBasicPropertiesLemma}
Let $K\neq\Delta_{[m]}$ be a simplicial complex on $[m]$ with $m\geq 3$. The following statements hold:
\begin{itemize}
\item[(a)] $\Bier(K)$ is a boundary of a simplex $\Leftrightarrow$ either $K=\partial \Delta_{[m]}$ or $K^\vee=\partial \Delta_{[m']}$;
\item[(b)] $\Bier(K)$ is a flag simplicial complex $\Leftrightarrow$ $K$ and $K^\vee$ are flag simplicial complexes;
\item[(c)] if $K=\Delta^d\ast L$, then $K^\vee=\Delta^d\ast L^\vee$ and 
$$
\Bier(K)=\Sigma^{d+1}\Bier(L)=\partial\diamondsuit^{d+1}\ast\Bier(L), 
$$
where $\diamondsuit^n$ denotes the $n$-dimensional cross-polytope; that is, $\diamondsuit^n := (I^n)^*$.
\end{itemize}
\end{lemma}
\begin{proof}
To prove (a) note that $\dim(\Bier(K))=m-2$ implies that $\Bier(K)$ is a boundary of a simplex if and only if $\Bier(K)$ is a boundary of $\Delta^{m-1}$. Then $\min(\Bier(K))$ equals either $\min(K)$ or $\min(K^\vee)$. This holds if and only if either $K=\partial \Delta_{[m]}$ or $K^\vee=\partial \Delta_{[m']}$. 

Statement (b) follows immediately from the structure of the set $\min(\Bier(K))$ and the definition of a flag simplicial complex.

To prove (c) we use induction on $d$; it suffices to prove it for $d=0$. Observe that $K$ is a cone over $L$ with apex $v$ implies that each maximal simplex of $K$ is the union of $\{v\}$ and a maximal simplex of $L$. By Alexander duality it is equivalent to $v'$ not belonging to any minimal non-face of $L^\vee$ and it means that each maximal simplex of $K^\vee$ contains $\{v'\}$. Thus, $K^\vee$ is a cone over $L^\vee$ with apex $v'$. Finally, the identity $\Bier(K)=\Sigma\Bier(L)$ follows by inspecting the sets of minimal non-faces on both sides. This finishes the proof.
\end{proof}

The next result was proven first in~\cite[Proposition 2.1]{HK}. Below we provide a different proof, which makes use of the classification of Bier spheres in dimensions not greater than two, for the sake of completeness.

\begin{lemma}\label{FlagBierClassificationLemma}
Let $K\neq\Delta_{[m]}$ be a simplicial complex on $[m]$ with $m\geq 3$. Then $\Bier(K)$ is flag if and only if it is combinatorially equivalent to a Bier sphere over a join of a simplex and one of the following complexes:
\begin{enumerate}
    \item a point and a ghost vertex;
    \item two disjoint points;
    \item disjoint vertex and segment;
    \item three disjoint points;
    \item chain on four vertices;
    \item cycle on four vertices.
\end{enumerate}
Furthermore, the Bier spheres corresponding to (1) and (2) are simplicially isomorphic. The same holds for the Bier spheres corresponding to (5) and (6).
\end{lemma}
\begin{proof}
Suppose $K$ has no ghost vertices. Since $\dim\Bier(K)=m-2$, by  Lemma~\ref{BierBasicPropertiesLemma} (b) $K^\vee$ is flag and $K$ must be a pure flag simplicial complex of dimension $m-3$. Let us use induction on $m$. 

For $m\leq 4$ the statement is true as shown in Example~\ref{BierOneDimExample} and~\cite[Theorem 2.16]{LS}, so consider $m\geq 5$ and let $\sigma$ be a maximal simplex in $K$. Then $|\sigma|=m-2$ and let $p,q$ be the two remaining vertices of $K$. 

If $\{p,q\}\notin K$, then each of $p$ and $q$ is linked by an edge with $m-3$ vertices of $\sigma$, since $K$ is pure. As $(m-3)+(m-3)>m-2=|\sigma|$, there is a vertex $i\in\sigma$, such that $\{i,p\}$ and $\{i,q\}$ are edges in $K$. This implies that $K$ is a cone with apex $i$ and we apply the inductive hypothesis.

If $\{p,q\}\in K$, then there is a maximal simplex $\tau\in K$ such that $\{p,q\}\subseteq\tau$. Since $|\tau|+|\sigma|=2m-4>m$, there is a common vertex $j\in \sigma\cap\tau$. Hence $K$ is a cone with apex $j$ and we apply the inductive hypothesis again.

Suppose $K^\vee$ has no ghost vertices. By Lemma~\ref{BierBasicPropertiesLemma} (c) if $K=\Delta^d\ast L$, then its Alexander dual is $K^\vee=\Delta^d\ast L^\vee$. Hence in this case the above argument applies. 

Suppose both $K$ and $K^\vee$ have ghost vertices; denote their sets of ghost vertices by $V$ and $V'$, respectively. Observe that by definition of Alexander duality, if $i\in V$, then $[m']\setminus\{i'\}\in K^\vee$ and therefore either $V'=\varnothing$, or $V'=\{i'\}$ . Thus, w.l.o.g. we can assume that in the remaining case under consideration we have $V=\{m\}, V'=\{m'\}$ and therefore $K=\Delta_{[(m-1)]}$, a join of a simplex with a complex on $\{m-1, m\}$ having the unique maximal face $\{m-1\}$ and the unique minimal non-face $\{m\}$, the ghost vertex. 

Finally, the required isomorphism of Bier spheres in (1) and (2), as well as in (5) and (6) follows from the definition and the classification of Bier spheres in dimension two given in~\cite[Theorem 2.16]{LS}. This finishes the proof.
\end{proof}

In what follows we use a special notation in case of a polytopal Bier sphere. Namely, when a Bier sphere $\Bier(K)$ is polytopal we denote by $P_K$ the corresponding convex simple polytope so that
$$
\Bier(K)=\partial P^*_K.
$$

\begin{theorem}\label{FlagBierPolytopeClassificationThm}
A Bier sphere $\Bier(K)$ is flag if and only if it is polytopal with $P_K$ being a flag nestohedron of one of the following types:
\begin{enumerate}
    \item $P_K=I^n$ for $n\geq 1$;
    \item $P_K=I^n\times P_5$ for $n\geq 0$;
    \item $P_K=I^n\times P_6$ for $n\geq 0$;
    \item $P_K=I^n\times Q_2^3$ for $n\geq 0$, 
\end{enumerate}
where $P_m$ for $m\geq 3$ denotes the $m$-gon and $Q_2^3$ denotes the 3-dimensional simple polytope, which can be obtained from the 3-dimensional cube $I^3$ by two adjacent edge cuts.
\end{theorem}
\begin{proof}
This follows directly from the classification of 1-dimensional Bier spheres, see Example~\ref{BierOneDimExample}, classification of 2-dimensional Bier spheres, see~\cite[Theorem 2.16]{LS}, Lemma~\ref{FlagBierClassificationLemma}, and an observation that if $K=K_{P_1}\ast K_{P_2}$, then $K=K_P$ for $P=P_1\times P_2$, where $P_1$ and $P_2$ are arbitrary simple polytopes, see~\cite[Example 2.2.9.4]{TT}. This finishes the proof.  
\end{proof}

The previous result solves~\cite[Problem 4.9]{LS} in the flag case. As far as we know, the general case remains open.

\begin{corollary}\label{NPconjectureBierCorollary}
The Nevo-Petersen conjecture holds in the class of flag Bier spheres.    
\end{corollary}
\begin{proof}
It remains to observe that all the simple polytopes in Theorem~\ref{FlagBierPolytopeClassificationThm} are flag nestohedra due to~\cite[Theorem 2.16]{LS} and~\cite[Lemma 1.5.20]{TT}, hence the statement follows from~\cite[Theorem 10]{AV}, since any flag nestohedron is a 2-truncated cube, see~\cite{BV}. Q.E.D.   
\end{proof}

\section{Murai spheres and the Nevo-Petersen conjecture}

In this section we classify all flag Murai spheres and show that the Nevo-Petersen conjecture holds for them. 

Let us start with the classification of all flag Murai spheres $\Bier_c(M)$ in dimensions less than two, that is the case when $|c|\leq 3$.

\begin{example}\label{OneDimFlagMuraiClassificationExample}
In case $|c|=1$ the only Murai sphere is the empty set. In case $|c|=2$ the only Murai sphere is the $0$-dimensional sphere $S^0$, which is a flag complex. Thus, suppose $|c|=3$. In the formulae below a disjoint union with an empty set means that a simplicial complex on $[3]$ has a ghost vertex.

Let $c=(3)$. Then one has the following list of flag Murai spheres:
\begin{itemize}
\item $M=\langle x\rangle=M^\vee$. Then
$$
\Bier_{c}(M)=\Bier_{c}(M^\vee)=\Bier(\Delta^1\sqcup\varnothing)=\partial P_4.
$$
\end{itemize}
Let $c=(2,1)$. Then one has the following list of flag Murai spheres:
\begin{itemize}
\item $M=\langle x\rangle$ and $M^\vee=\langle x^2, y\rangle$. Then 
$$
\Bier_{c}(M)=\Bier_{c}(M^\vee)=\Bier(\ast\sqcup \varnothing\sqcup \varnothing)=\partial P_4;
$$ 
\item $M=\langle y\rangle$ and $M^\vee=\langle xy\rangle$. Then 
$$
\Bier_{c}(M)=\Bier_{c}(M^\vee)=\Bier(\Delta^1\sqcup\varnothing)=\partial P_4;
$$ 
\item $M=\langle x, y\rangle=M^\vee$. Then 
$$
\Bier_{c}(M)=\Bier_{c}(M^\vee)=\Bier(\ast\sqcup\ast\sqcup\varnothing)=\partial P_5;
$$ 
\end{itemize}
Let $c=(1,1,1)$. Then one has the following list of flag Murai spheres:
\begin{itemize}
\item $M=\langle x\rangle$ and $M^\vee=\langle xy, xz\rangle$. Then 
$$
\Bier_{c}(M)=\Bier_{c}(M^\vee)=\Bier(\ast\sqcup\varnothing\sqcup\varnothing)=\partial P_4;
$$
\item $M=\langle xy\rangle=M^\vee$. Then 
$$
\Bier_{c}(M)=\Bier_{c}(M^\vee)=\Bier(\Delta^1\sqcup\varnothing)=\partial P_4;
$$
\item $M=\langle x, y\rangle$ and $M^\vee=\langle z, xy\rangle$. Then 
$$
\Bier_{c}(M)=\Bier_{c}(M^\vee)=\Bier(\ast\sqcup\ast\sqcup\varnothing)=\partial P_5;
$$
\item $M=\langle x, y, z\rangle=M^\vee$. Then
$$
\Bier_{c}(M)=\Bier_{c}(M^\vee)=\Bier(\ast\sqcup\ast\sqcup\ast)=\partial P_6.
$$
\end{itemize}
Thus, the only 1-dimensional flag Murai spheres are the boundaries of either a 4-gon, or a 5-gon, or a 6-gon.
\end{example}

Recall that given a monomial ideal $I$ in the polynomial algebra $\Bbbk[m]$, we denote by $G(I)$ the unique minimal set of monomial generators of $I$. The next result is a useful tool in the classification of high-dimensional flag Murai spheres. 

\begin{lemma}\label{IdealGeneratorsLemma}
For any proper $c$-multicomplex $M$ one has:
$$
G(I_c(M))=\{(x^a)^c\,|\,x^a\in\max(M^\vee)\}=\min(M)
$$
and 
$$
G(I_c(M)^\vee)=\{(x^a)^c\,|\,x^a\in\max(M)\}=\min(M^\vee).
$$
\end{lemma}
\begin{proof}
The two identities above are equivalent, since 
$$
(M^\vee)^\vee=M\text{ and }I_c(M^\vee)=I_c(M)^\vee.
$$
The first identity follows directly from the definitions. Q.E.D. 
\end{proof}

Now we are ready to state and prove the classification of flag Murai spheres in dimensions greater than one, when they differ from Bier spheres.

\begin{lemma}\label{FlagMuraiMainCaseLemma}
Let $M$ be a proper $c$-multicomplex with $c\in \N^m$ and $|c|>m, |c|\geq 4$. Then $\Bier_c(M)$ is flag if and only if it is combinatorially equivalent to a Murai sphere over a multicomplex generated by one of the following sets of monomials in $\Bbbk[m]$: 
\begin{enumerate}
\item $\{(x_i^{c_i})^c\}$ and $c_i=2$ for some $1\leq i\leq m$, $c_j=1$ for all $1\leq j\neq i\leq m$;

\item $\{(x_i^{c_{i}-1})^c\}$ and $c_i=3$ for some $1\leq i\leq m$, $c_p=1$ for all $1\leq p\neq i\leq m$;

\item $\{(x_i^{c_i})^c, (x_jx_i)^c\}$ and $c_i=2$ for some $1\leq i\leq m$, $c_p=1$ for all $1\leq p\neq i\leq m$;

\item $\{(x_i^{c_i})^c, (x_kx_j)^c\}$ and $c_i=2$ for some $1\leq i\leq m$, $c_p=1$ for all $1\leq p\neq i\leq m$;

\item $\{(x_i^{c_i})^c, (x_j^{c_j})^c\}$ and $c_i=2$, $c_j\leq 2$ for some $1\leq i\neq j\leq m$, $c_p=1$ for all $1\leq p\neq i,j\leq m$.
\end{enumerate}
\end{lemma}
\begin{proof}
Suppose either $M$ or $M^\vee$ has a ghost variable. W.l.o.g. we can assume $x_m\notin M$. 

Then $x_{m}\in G(I_c(M))$ and 
$$
(x_m)^c = (x_1^{c_1}\cdot\ldots\cdot x_{m-1}^{c_{m-1}}x_m^{c_{m}-1})\in \max(M^\vee).
$$

Since $1\in M$, $M^\vee$ is proper and it follows that 
$$
p(x_1,\ldots,x_{m-1})x_m^{c_m}\in G(I_c(M^\vee))
$$
for a certain monomial $p(x_1,\ldots,x_{m-1})$. The degree of its $\ast$-polarization is equal to $\deg(p) + c_m$ and hence, as $\Bier_c(M)$ is flag, this yields that either $I_c(M^\vee)=(x_m^2), c_m=2$ or $I_c(M^\vee)=(x_m), c_m=1$, or $p=x_i, c_m=1$ for some $1\leq i\leq m-1$.

In the first case, that is:
$$
I_c(M^\vee)=(x_m^2), c_m=2
$$
one has:
$$
x_1^{c_1}\cdot\ldots\cdot x_{m-1}^{c_{m-1}}=(x_m^2)^c\in \max(M),
$$
that is, since $x_m\notin M$, 
$$
\max(M)=\{x_1^{c_1}\cdot\ldots\cdot x_{m-1}^{c_{m-1}}\}\text{ and }c_m=2.
$$ 
Therefore 
$$
I_c(M)=(x_m)
$$
and 
$$
\max(M^\vee)=\{x_1^{c_1}\cdot\ldots\cdot x_{m-1}^{c_{m-1}}x_m\}.
$$ 
Since $\Bier_c(M)$ is flag, polarizations of $x_i^{c_i+1}$ for $1\leq i\leq m-1$ must have degree $2$, therefore 
$$
c_i=1\text{ for }1\leq i\leq m-1\text{ and }c_m=2.
$$
The statement holds.

In the second case, that is:
$$
I_c(M^\vee)=(x_m)
$$
we have
$$
\max(M^\vee)=\{x_1^{c_1}\cdot\ldots\cdot x_{m-1}^{c_{m-1}}\}.
$$
Therefore
$$
I_c(M)=(x_m)
$$
and 
$$
\max(M)=\{x_1^{c_1}\cdot\ldots\cdot x_{m-1}^{c_{m-1}}\}.
$$
Since $\Bier_c(M)$ is flag, polarizations of $x_i^{c_i+1}$ for $1\leq i\leq m-1$ must have degree $2$, therefore 
$$
c_i=1\text{ for }1\leq i\leq m.
$$

In the third case, that is:
$$
p=x_{i}, c_m=1
$$
one has:
$$
(x_{i}x_m)^c\in \max(M).
$$

It follows that
$$
I_c(M)=(x_m,(x_m)^c)
$$
we have
$$
\max(M)=\{x_1^{c_1}\cdot\ldots x_i^{c_{i}-1}\cdot x_{m-1}^{c_{m-1}}\}
$$
and therefore 
$$
\deg(x_1^{c_1}\cdot\ldots\cdot x_{m-1}^{c_{m-1}})=|c|-c_m=|c|-1\leq 2,  
$$
hence
$$
|c|\leq 3, c_m=1.
$$

Now, assume neither $M$, nor $M^\vee$ has ghost variables.
Suppose $c_m\geq 2$. 

First, suppose $x_m^{c_m}\notin M$ and $x_m^{c_m}\notin M^\vee$. 

There exists $x_m^{a}\in\min(M), a\geq 2$ ($x_m$ is not a ghost variable) and $\min(M)=G(I_c(M))$. Hence $a=2$ and 
$$
x_m^2\in G(I_c(M))
$$
and similarly,
$$
x_m^2\in G(I_c(M^\vee))
$$

It implies that 
$$
(x_m^2)^c\in\max(M^\vee)
$$ 
and similarly,
$$
(x_m^2)^c\in\max(M).
$$ 

Since $\Bier_c(M)$ is flag, it follows that $c_m\leq 3$.

In case $c_m=3$, we get: 
$$
I_c(M^\vee)=(x_m^2).
$$
Thus,
$$
\max(M)=\{x_1^{c_{1}}\ldots x_{m-1}^{c_{m-1}}x_m\}
$$
and 
$$
I_c(M)=(x_m^2).
$$
Because of the polarizations of $x_i^{c_{i}+1}$ in the Stanley-Reisner ideal of the flag sphere $\Bier_c(M)$ one has
$$
c_i=1\text{ for all }1\leq i\leq m-1.
$$

In case $c_m=2$, we get:
$$
(x_m^2)^c=x_1^{c_{1}}\ldots x_{m-1}^{c_{m-1}}\in\max(M^\vee)
$$
and similarly
$$
(x_m^2)^c=x_1^{c_{1}}\ldots x_{m-1}^{c_{m-1}}\in\max(M).
$$

If $I_c(M)=(x_m^2)$, then 
$$
x_1^{c_{1}}\ldots x_{m-1}^{c_{m-1}}x_m\in M,
$$
which contradicts
$$
(x_m^2)^c=x_1^{c_{1}}\ldots x_{m-1}^{c_{m-1}}\in\max(M).
$$

If $x_ix_m\in G(I_c(M))$, then the elements in $\max(M)$ dividing $x_m$ must have degree $\leq |c|-1-c_i$ and $\geq |c|-2$, hence $c_i=1$. Therefore, if 
$$
I_c(M)=(x_ix_m, x_m^2),
$$
then $c_i=1, c_m=2$ and 
$$
\max(M)=\{(x_m^2)^c, (x_ix_m)^c\}.
$$
In this case, 
$$
I_c(M)=I_c(M^\vee).
$$
Because of the polarizations of $x_i^{c_{i}+1}$ in the Stanley-Reisner ideal of the flag sphere $\Bier_c(M)$ one has
$$
c_i=1\text{ for all }1\leq i\leq m-1.
$$

Now, suppose w.l.o.g. that $x_m^{c_m}\in M$, then polarization of $x_m^{c_m+1}$ must divide an element in $G(I_c(M))\cup G(I_c(M^\vee))$ and therefore, by Lemma~\ref{IdealGeneratorsLemma}:
$$
(x_m^{c_m})^c\in M.
$$

In case 
$$
(x_m^{c_m})^c\in \max(M),
$$
its degree must be $|c|-c_m\geq |c|-2$ and so $c_m\leq 2$. Thus 
$$
x_m^{c_m}\in M\text{ and } c_m=2.
$$

If $x_ix_m\in G(I_c(M))$, then the elements in $\max(M)$ dividing $x_m$ must have degree $\leq |c|-c_i$ and $\geq |c|-2$, hence $c_i\leq 2$. Therefore, if 
$$
I_c(M)=(x_ix_m, x_jx_m),
$$
then $c_i=c_j=1, c_m=2$ and 
$$
\max(M)=\{(x_m^2)^c, (x_ix_j)^c\}.
$$
Therefore,
$$
I_c(M^\vee)=(x_m^2, x_ix_j).
$$
Because of the polarizations of $x_i^{c_{i}+1}$ in the Stanley-Reisner ideal of the flag sphere $\Bier_c(M)$ one has
$$
c_k=1\text{ for all }1\leq k\leq m-1.
$$

If 
$$
I_c(M)=(x_ix_m),
$$
then $c_i\leq 2$, $c_m=2$ and 
$$
\max(M)=\{(x_m^2)^c, (x_i^{c_i})^c\}.
$$
Therefore,
$$
I_c(M^\vee)=(x_m^2, x_i^{c_i}).
$$
Because of the polarizations of $x_i^{c_{i}+1}$ in the Stanley-Reisner ideal of the flag sphere $\Bier_c(M)$ one has
$$
c_k=1\text{ for all }1\leq k\neq i\leq m-1.
$$

In case 
$$
x_1^{c_1}\ldots x_{m-1}^{c_{m-1}}x_m^{a_m}\in\max(M)\text{ for some }a_m\geq 1,
$$
all elements in $G(I_c(M))$ divide the monomials of degree $\geq 3$ of the type $x_m^{a_{m}+1}x_i$ for certain $1\leq i\leq m-1$, because $x_m^{c_m}\in M$, which contradicts the flag-ness of $\Bier_c(M)$. 
\end{proof}

Taking into account all the previously obtained classification results, we get the next general statement. When a Murai sphere $\Bier_c(M)$ is polytopal we denote by $P_{M,c}$ the corresponding convex simple polytope so that
$$
\Bier_c(M)=\partial P^*_{M,c}.
$$

\begin{theorem}\label{FlagMuraiPolytopeClassificationThm}
A Murai sphere $\Bier_c(M)$ is flag if and only if it is polytopal with $P_{M,c}$ being a flag nestohedron of one of the following types:
\begin{enumerate}
    \item $P_{M, c}=I^n$ for $n\geq 1$;
    \item $P_{M, c}=I^n\times P_5$ for $n\geq 0$;
    \item $P_{M, c}=I^n\times P_6$ for $n\geq 0$;
    \item $P_{M, c}=I^n\times Q_2^3$ for $n\geq 0$, 
\end{enumerate}
where $P_m$ for $m\geq 3$ denotes the $m$-gon and $Q_2^3$ denotes the 3-dimensional simple polytope, which can be obtained from the 3-dimensional cube $I^3$ by two adjacent edge cuts.
\end{theorem}
\begin{proof}
Let $\Bier_c(M)$ be a flag Murai sphere. If either $|c|\leq 3$, or $c=(1,\ldots,1)\in\N^m$ with $m\geq 2$, then the statement is true due to Example~\ref{OneDimFlagMuraiClassificationExample} and Theorem~\ref{FlagBierPolytopeClassificationThm}. 

Suppose $c\in\N^m$ with $|c|\geq 4$ and $|c|>m$. By Lemma~\ref{FlagMuraiMainCaseLemma} we have the following cases.

\begin{enumerate}
\item $\max(M)=\{(x_i^{c_i})^c\}$ and $c_i=2$ for some $1\leq i\leq m$, $c_j=1$ for all $1\leq j\neq i\leq m$;

In this case the Stanley-Reisner ideal has the form:
$$
(x_{i,0})+(x_{i,2}x_{i,1})+(x_{1,0}x_{1,1},\ldots,x_{i-1,0}x_{i-1,1},x_{i+1,0}x_{i+1,1},\ldots,x_{m,0}x_{m,1}).
$$

This Murai sphere is the boundary of the cross-polytope with the vertex set $V=\tilde{X}\setminus\{x_{i}^{(0)}\}$.

\item $\max(M)=\{(x_i^{c_{i}-1})^c\}$ and $c_i=3$ for some $1\leq i\leq m$, $c_p=1$ for all $1\leq p\neq i\leq m$;

In this case the Stanley-Reisner ideal has the form:
$$
(x_{i,0}x_{i,1})+(x_{i,3}x_{i,2})+(x_{1,0}x_{1,1},\ldots,x_{i-1,0}x_{i-1,1},x_{i+1,0}x_{i+1,1},\ldots,x_{m,0}x_{m,1}).
$$

This Murai sphere is the boundary of the cross-polytope with the vertex set $V=\tilde{X}$.

\item $\max(M)=\{(x_i^{c_i})^c, (x_jx_i)^c\}$ and $c_i=2$ for some $1\leq i\leq m$, $c_p=1$ for all $1\leq p\neq i\leq m$;

In this case the Stanley-Reisner ideal has the form:
$$
(x_{i,0}x_{j,0},x_{i,0}x_{i,1})+(x_{i,2}x_{j,1},x_{i,2}x_{i,1})+
(x_{1,0}x_{1,1},\ldots,x_{i-1,0}x_{i-1,1},x_{i+1,0}x_{i+1,1},\ldots,x_{m,0}x_{m,1}).
$$

This Murai sphere is the iterated $(m-2)$-times suspension over $\Bier(\langle \{x_{i}^{(0)}\}, \{x_{i}^{(1)}\}\rangle\sqcup\{\varnothing\})$, which is a 5-gon with the vertex set 
$$
\{x_{i}^{(0)},x_{i}^{(1)},x_{i}^{(2)},x_{j}^{(0)},x_{j}^{(1)}\}.
$$

\item $\max(M)=\{(x_i^{c_i})^c, (x_kx_j)^c\}$ and $c_i=2$ for some $1\leq i\leq m$, $c_p=1$ for all $1\leq p\neq i\leq m$;

In this case the Stanley-Reisner ideal has the form:
$$
(x_{i,0}x_{j,0},x_{i,0}x_{k,0})+(x_{i,2}x_{i,1},x_{j,1}x_{k,1})+
(x_{1,0}x_{1,1},\ldots,x_{i-1,0}x_{i-1,1},x_{i+1,0}x_{i+1,1},\ldots,x_{m,0}x_{m,1}).
$$

This Murai sphere is the iterated $(m-3)$-times suspension over $\Bier(\langle\{x_{i}^{(0)},x_{i}^{(2)}\},\{x_{k}^{(0)},x_{i}^{(2)}\}\rangle\sqcup\{\varnothing\})$, which is a suspension over the 5-gon and has the vertex set
$$
\{x_{i}^{(0)},x_{i}^{(1)},x_{i}^{(2)},x_{j}^{(0)},x_{j}^{(1)},x_{k}^{(0)},x_{k}^{(1)}\}.
$$

\item $\max(M)=\{(x_i^{c_i})^c, (x_j^{c_j})^c\}$ and $c_i=2$, $c_j=2$ for some $1\leq i\neq j\leq m$, $c_p=1$ for all $1\leq p\neq i,j\leq m$.

In this case the Stanley-Reisner ideal has the form:
$$
(x_{i,0}x_{j,0})+(x_{i,2}x_{i,1},x_{j,2}x_{j,1})+
(x_{1,0}x_{1,1},\ldots,x_{i-1,0}x_{i-1,1})+
$$
$$
+(x_{i+1,0}x_{i+1,1},\ldots,x_{j-1,0}x_{j-1,1},x_{j+1,0}x_{j+1,1},\ldots x_{m,0}x_{m,1})
$$

This Murai sphere is the boundary of a cross-polytope with the vertex set $\tilde{X}$. 

\item $\max(M)=\{(x_i^{c_i})^c, (x_j^{c_j})^c\}$ and $c_i=2$, $c_j=1$ for some $1\leq i\neq j\leq m$, $c_p=1$ for all $1\leq p\neq i,j\leq m$.

In this case the Stanley-Reisner ideal has the form:

$$
(x_{i,0}x_{j,0})+(x_{i,2}x_{i,1},x_{j,1})+
(x_{1,0}x_{1,1},\ldots,x_{i-1,0}x_{i-1,1},x_{i+1,0}x_{i+1,1},\ldots x_{m,0}x_{m,1}).
$$

This Murai sphere is the boundary of a cross-polytope with the vertex set $\tilde{X}\setminus\{x_{j}^{(1)}\}$. 
\end{enumerate}

It follows that the statement is also true in the general case. This finishes the proof. 
\end{proof}

\begin{corollary}\label{NPconjectureMuraiCorollary}
The Nevo-Petersen conjecture holds in the class of flag Murai spheres.
\end{corollary}
\begin{proof}
The proof is the same as the one for Bier spheres, see Corollary~\ref{NPconjectureBierCorollary}. Q.E.D.    
\end{proof}

\section{Bier spheres and face rings}

In this section we recall the notions of a Golod and a minimally non-Golod ring in the case of Stanley-Reisner rings and study the Golod and minimally non-Golod properties of face rings of Bier spheres. 

Let us start with the following key construction.

\begin{construction}
If $K$ is a simplicial complex on $[m]$ and $A=\Bbbk[K]$, then the \emph{Koszul complex} of $A$ is defined to be the differential graded algebra
$$
[K_A,d] = [\Bbbk[K]\otimes_{\Bbbk}\Lambda[u_{1},\ldots,u_{m}],d],
$$
where 
$$
\deg(u_i)=1, \deg(v_i)=2\text{ and }d(u_{i})=v_{i}, d(v_i)=0,\text{ for all }1\leq i\leq m.
$$

Therefore, the Koszul complex of $\Bbbk[K]$ looks as follows:
$$
0\to\Lambda^m[u_{1},\ldots,u_{m}]\otimes_{\Bbbk}\Bbbk[K]\to\ldots\to\Lambda^1[u_{1},\ldots,u_{m}]\otimes_{\Bbbk}\Bbbk[K]\to \Bbbk[K]\to 0,
$$
where $\Lambda^i[u_{1},\ldots,u_{m}]$ is a subspace in $\Lambda[u_{1},\ldots,u_{m}]$ generated by (exterior) monomials of degree $i, 1\leq i\leq m$ and the maps are the $d$ differentials.

Finally, the {\emph{Koszul homology}} or the {\emph{Tor-algebra}} of $\Bbbk[K]$ is the cohomology of the Koszul complex:
$$
\Tor_{\Bbbk[m]}^{\ast}(\Bbbk[K],\Bbbk):=H^{\ast}[\Bbbk[K]\otimes_{\Bbbk}\Lambda[u_{1},\ldots,u_{m}],d].
$$
\end{construction}

The properties of Koszul homology we are especially interested in in this paper are Golodness and minimal non-Golodness. To define these two properties, we need the following notion. Given a simplicial complex $K$ on $[m]$ and $i\in [m]$ we denote by $\del_K(i)$ the full subcomplex in $K$ on $[m]\setminus\{i\}$ and call it the \emph{deletion} of $i$ from $K$.

\begin{definition}
Let $K$ be a simplicial complex with $m$ vertices. Its face ring $\Bbbk[K]$ is called \emph{Golod} if multiplication and all higher Massey products vanish in $\Tor^{+}_{\Bbbk[m]}(\Bbbk[K],
\Bbbk)$. The ring $\Bbbk[K]$ is called \emph{minimally non-Golod} if it is not Golod and the ring $\Bbbk[\del_K(i)]$ is Golod for each $i\in [m]$. We say that $K$ is a \emph{Golod complex} (\emph{minimally non-Golod complex}) if its face ring $\Bbbk[K]$ is Golod (minimally non-Golod) for any field $\Bbbk$, respectively. 
\end{definition}

\begin{remark}
Note that Koszul homology of a minimally non-Golod face ring may contain a non-trivial Massey product, see~\cite[Theorem 2]{LP}.      
\end{remark}

In our classification of minimally non-Golod Bier spheres we are going to make use of the next fundamental result, showing that Koszul homology of a face ring is isomorphic to the direct sum of reduced simplicial cohomology groups of all full subcomplexes. We refer to~\cite{H} as well as \cite[Theorem 3.2.9, Proposition 3.2.10]{TT} for the details.

\begin{theorem}[\cite{H,TT}]\label{BPtheorem}
Let $K$ be a simplicial complex on $[m]$. Then an isomorphism of $\Bbbk$-modules holds:
$$
\Tor^{\ast}_{\Bbbk[m]}(\Bbbk[K],\Bbbk)\cong\bigoplus\limits_{-i+2j=\ast}
\bigoplus\limits_{J\subseteq [m],\,|J|=j}\tilde{H}^{j-i-1}(K_J;\Bbbk).
$$
The product in the r.h.s. induced
by the  above isomorphism coincides up to a sign with the cohomology product induced by the maps of simplicial
cochains
\begin{equation}\label{fullsubcochain}
\begin{aligned}
  \mu\colon C^{p-1}(K_I)\otimes C^{q-1}(K_J)&\longrightarrow C^{p+q-1}(K_{I\cup
  J}),\\
  \alpha_L\otimes \alpha_M\ \ &\longmapsto
    \left\{%
    \begin{array}{ll}
    \alpha_{L\sqcup M}, & \text{if $I\cap J=\varnothing$;} \\
    0, & \hbox{otherwise.} \\
    \end{array}%
    \right.
\end{aligned}
\end{equation}
Here $\alpha_{L\sqcup M}\in C^{p+q-1}(K_{I\sqcup J})$ denotes
the basis simplicial cochain corresponding to $L\sqcup M$ if the
latter is a simplex of $K_{I\sqcup J}$ and zero otherwise.
\end{theorem}

Now we are ready to state and prove the classification of Golod and minimally non-Golod Bier spheres.

\begin{theorem}
Let $K\neq\Delta_{[m]}$ be a simplicial complex on $[m]$ with $m\geq 3$. The following statements hold.
\begin{itemize}
\item[(a)] $\Bier(K)$ is Golod if and only if either $K=\partial\Delta_{[m]}$ or $K^\vee=\partial\Delta_{[m']}$. In this case $\Bier(K)$ is polytopal and we have
$$
P_K=\Delta^{m-1};
$$
\item[(b)] $\Bier(K)$ is minimally non-Golod if and only if either $K$ or $K^\vee$ equals the set of $\ell$ disjoint points, where $1\leq\ell\leq m$. In this case $\Bier(K)$ is polytopal and $P_K$ is a truncation polytope obtained from $\Delta^{m-1}$ by cutting off $\ell$ of its vertices; that is
$$
P_K=\vc^{\ell}(\Delta^{m-1}).
$$
\end{itemize}
\end{theorem}
\begin{proof}
To prove statement (a) note that a triangulated sphere is Golod if and only if it is a boundary of a simplex, since otherwise Koszul homology of its face ring has a non-trivial product. Then Lemma~\ref{BierBasicPropertiesLemma} (a) finishes the proof.

Let us prove statement (b). For $m=3$ it follows from Example~\ref{BierOneDimExample} and for $m=4$ it follows from~\cite[Theorem 2.16]{LS}. Therefore let $m\geq 5$ and $\Bier(K)$ be minimally non-Golod. Note that 
$$
f_0(\Bier(K))\geq m\geq 5.
$$

Suppose $\sk^1(K)$ contains an edge. W.l.o.g. we can assume that $\{1,2\}\in K$. If $\{1',2'\}\in K^\vee$ the deletion complex $\del_{\Bier(K)}(i)$ contains a 4-cycle labeled clockwise by $\{1,2,1',2'\}$ as a full subcomplex, for any $i\in [m]\setminus\{1,2\}$. By formula (\ref{fullsubcochain}), the face ring $\Bbbk[\del_{\Bier(K)}(i)]$ has a non-trivial product in Koszul homology and hence is non-Golod. Therefore $\{1',2'\}\notin K^\vee$.

Observe that if $\{i\}\notin K$ for $i\geq 3$, then by definition of Alexander duality we have $\{1',2'\}\subset [m']\setminus\{i'\}\in K^\vee$, whereas $\{1',2'\}\notin K^\vee$. This contradiction yields $m=f_0(K)$; in other words, $K$ has no ghost vertices.

Let us show that for each $i, 3\leq i\leq m$ there are edges $\{1,i\}$ and $\{2,i\}$ in $K$.

If there exists an $i, 3\leq i\leq m$ such that both $\{1,i\}, \{2,i\}\notin K$, then the simplicial complex $\del_{\Bier(K)}(j)$ contains a 5-cycle labeled clockwise by $\{1,2,1',i,2'\}$ as a full subcomplex, for any $j\in [m]\setminus\{1,2,i\}$. 
By formula (\ref{fullsubcochain}), the face ring $\Bbbk[\del_{\Bier(K)}(j)]$ has a non-trivial product in Koszul homology and hence is non-Golod, a contradiction. 

If there exists an $i, 3\leq i\leq m$ such that $\{1,i\}\notin K$ and $\{2,i\}\in K$, then the simplicial complex $\del_{\Bier(K)}(1')$ contains a 4-cycle labeled clockwise by $\{1,2,i,2'\}$ as a full subcomplex. 
By formula (\ref{fullsubcochain}), the face ring $\Bbbk[\del_{\Bier(K)}(1')]$ has a non-trivial product in Koszul homology and hence is non-Golod, a contradiction. 

If there exists an $i, 3\leq i\leq m$ such that $\{1,i\}\in K$ and $\{2,i\}\notin K$, then the simplicial complex $\del_{\Bier(K)}(2')$ contains a 4-cycle labeled clockwise by $\{1,2,1',i\}$ as a full subcomplex.
By formula (\ref{fullsubcochain}), the face ring $\Bbbk[\del_{\Bier(K)}(2')]$ has a non-trivial product in Koszul homology and hence is non-Golod, a contradiction. 

Now, let us show that for each $i, j$ such that $3\leq i\neq j\leq m$ there is an edge $\{i,j\}\in K$.

Indeed, if $\{i,j\}\notin K$, then the simplicial complex $\del_{\Bier(K)}(2)$ contains a 4-cycle labeled clockwise by $\{1,i,1',j\}$ as a full subcomplex. By formula (\ref{fullsubcochain}), the face ring $\Bbbk[\del_{\Bier(K)}(2)]$ has a non-trivial product in Koszul homology and hence is non-Golod, a contradiction. 
 
It follows that $K^\vee$ contains no edges and therefore $K^\vee$ is a disjoint set of points. 

Indeed, as was shown above $\sk^1(K)$ is a complete graph on $m$ vertices. Suppose an edge $\{i',j'\}\in K^\vee$. Then the simplicial complex $\del_{\Bier(K)}(v)$ contains a 4-cycle labeled clockwise by $\{i,j,i',j'\}$ as a full subcomplex, for any $v\in [m]\setminus\{i,j\}$.
By formula (\ref{fullsubcochain}), the face ring $\Bbbk[\del_{\Bier(K)}(v)]$ has a non-trivial product in Koszul homology and hence is non-Golod, a contradiction. 

To prove the opposite implication, it remains to observe that if 
$$
V(K^\vee)=\{1',2',\ldots,\ell'\},
$$
then $\Bier(K)$ is obtained from $\partial\Delta_{[m]}$ by stellar subdivisions of the facets $[m]\setminus\{i\}$ by vertices $i'$ for $1\leq i\leq \ell$. We obtain a polytopal sphere whose Bier polytope $P_K$ is of the type $\vc^{\ell}(\Delta^{m-1})$. 

Nerve complexes of truncation polytopes with $\ell\geq 1$ are minimally non-Golod by~\cite{BJ}. This finishes the proof.
\end{proof}

\begin{remark}
Note that the combinatorial type of a truncation polytope $\vc^\ell(\Delta^{m-1})$ with $\ell\geq 3$ and $m\geq 4$ depends on the vertices that we cut. Hence, to be more precise, our notation $P=\vc^\ell(\Delta^{m-1})$ reflects the fact that $P$ has the combinatorial type of a certain concrete truncation polytope obtained from $\Delta^{m-1}$ by a sequence of $\ell$ vertex cuts as described in the proof of the previous theorem.  
\end{remark}

The previous result solves~\cite[Problem 4.8]{LS} in the case of truncation polytopes of the type $\vc^k(\Delta^n)$ with $0\leq k\leq n+1$. As far as we know, the general case remains open.

\section{Polyhedral products and Alexander duality}\label{sec:poly-prod}

Let $(X, A)$ be a pair of spaces and let $K$ be an abstract simplicial complex, $K\subseteq 2^{[m]}$.

\medskip

The associated {\em polyhedral roduct} (or, \emph{generalized moment-angle complex}, or \emph{$K$-power}) is the space
\begin{equation}\label{eqn:PP-1}
(X, A)^K = \mathcal{Z}_K(X, A) = \bigcup_{\sigma\in K} (X, A)^\sigma =
\bigcup_{\sigma\in K}( \prod_{i\in\sigma} X \times
\prod_{j\notin\sigma} A ) \subseteq X^m .
\end{equation}

\begin{figure}[htb]
    \centering
    \includegraphics[scale=0.6]{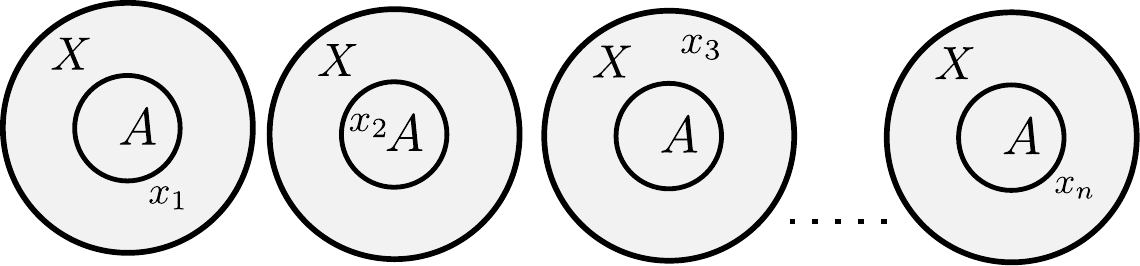}
    \caption{A ``random element'' of $X^m$.}
    \label{fig:darts}
\end{figure}

There is a quite useful, alternative definition (description) of the polyhedral product (\ref{eqn:PP-1}). Given a ``random element''
$x = (x_i)\in X^m$ (or alternatively a function $x : [m] \rightarrow X$), the corresponding ``hitting set'' $H_A(x)$ is the set
\[
 H_A(x) :=  \{i\in [m] \mid x_i\in A \} = x^{-1}(A)\, .
\]
Similarly, the associated complementary ``missing set''  $M_A(x) = H_A(x)^c$ is the set
\[
M_A(x) := \{i\in [m] \mid x_i\notin A \} = x^{-1}(A^c) \, .
\]

\begin{proposition}\label{prop:magic} For each $x\in X^m$
  \[
      x\in \mathcal{Z}_K(X, A) \quad \Leftrightarrow \quad M_A(x) \in K \quad \Leftrightarrow   \quad  H_A(x)\in W := 2^{[m]}\setminus K^\vee \, .
  \]
\end{proposition}\label{prop:Darts-1}

\subsection{Alexander duality revisited}

The following result will be referred to as the \emph{Alexander duality for generalized moment-angle complexes}.

\begin{proposition}[\cite{GW}] \label{prop:Grujic-Welker}
The polyhedral products $\mathcal{Z}_K(X, A)$ and $\mathcal{Z}_{K^\vee}(X, A^c)$ are disjoint (as subsets of $X^m$). Moreover,
\begin{equation}\label{eqn:Gru-Wel}
       \mathcal{Z}_K(X, A) \cup \mathcal{Z}_{K^\vee}(X, A^c)  =  X^m \, .
\end{equation}
\end{proposition}

\proof
For each $x\in X^m$ either $M_A(x) \in K$ or $M_{A^c}(x) \in K^\vee$, but not both!
Indeed, $M_A(x)\cap M_{A^c}(x) = \emptyset$ and $M_A(x)\cup M_{A^c}(x) = [m]$ and, as a consequence, for each $x\in X^m$ precisely
one of the relations $$x\in \mathcal{Z}_K(X, A), \quad x\in  \mathcal{Z}_{K^\vee}(X, A^c)$$ is satisfied. \qed

\section{Canonical cubulations of simplicial complexes}\label{sec:canonical_cubulation}

Here we summarize some basic facts about ``canonical cubulations'' (canonical quadrangulations) of simplicial complexes, needed in Section \ref{sec:meet-1}.
For more details the reader is referred to \cite[Section 2.9]{TT}, see also \cite{Zivaljevic15}.

\begin{figure}[hbt]
\centering
\includegraphics[scale=0.60]{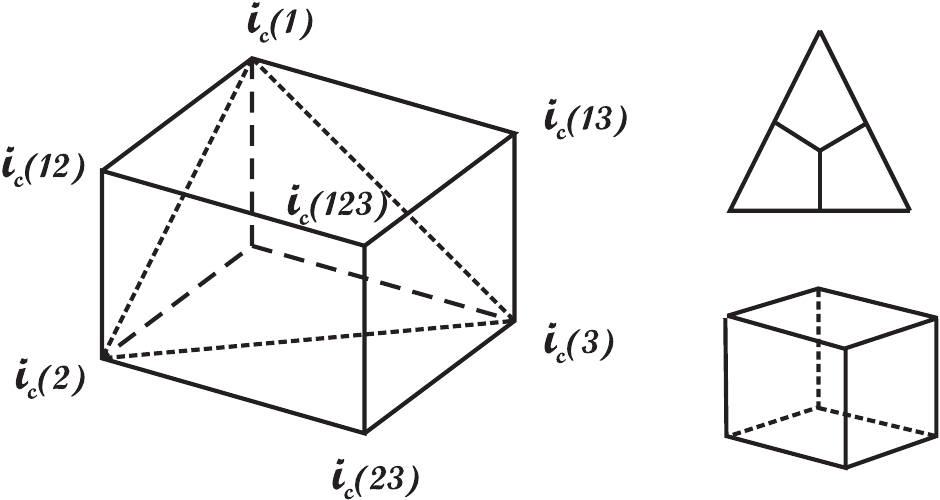}
\caption{The front complex of a cube.} \label{fig:front-and-back}
\end{figure}

Each non-empty subset $I = \{j_1,\ldots, j_k\}\subset [m]$ can be
associated both a vertex $b_I = (1/k)(e_{j_1}+\ldots +e_{j_k})$ of
the barycentric subdivision of the simplex $\Delta^{m-1}={\rm conv}\{e_1,\ldots, e_m\}$ and the vertex $e_I =
e_{j_1}+\ldots +e_{j_k}$ of the standard cube $I^m\subset
\mathbb{R}^m$. The correspondence $b_I \mapsto e_I$ is just the radial projection from the vertex $e_\emptyset = 0$ of the cube. Therefore it is extended to
a piecewise linear embedding $i_c : \Delta^{m-1} \rightarrow I^m$
which turns out to be an isomorphism of $\Delta^{m-1}$ with the
``front region'' $Front(I^m)$ of the cube $I^m$
(Figure~\ref{fig:front-and-back}).

\noindent
(The front region $Front(I^m)$ is formally described as the union of
all facets of $I^m$ which contain the vertex $e_{[m]}=e_1+\ldots
+e_m$.) 
Note that the map $i_c$ is linear (affine) on the top-dimensional simplices of the first
barycentric subdivision of $\Delta^{m-1}$.

\begin{definition}{\rm (Canonical cubulation)}\label{def:canonical_cubulation}
 The front region $Front(I^m)$ is a cubical subcomplex of $I^m$. The corresponding cubical complex induced on $\Delta^{m-1}$ by $i_c$ is
 called the \emph{canonical cubulation} of $\Delta^{m-1}$. More generally, for each abstract simplicial complex $K\subseteq 2^{[m]}$, the map $i_c$
 induces a \emph{canonical cubulation} of its geometric realization $\vert K\vert\subseteq \Delta^{m-1}$.
\end{definition}

Bier spheres $\Bier(K) = B(K,K^\vee)\subset 2^{[m]\cup [m']}$ have two types of vertices, however we can still talk about canonical cubulations. For this we need slightly more general geometric realizations of abstract simplicial complexes.

Given a simplicial complex $L\subseteq 2^S$ and a collection of vectors $\widehat{S}=\{v_i\}_{i\in S}$ labeled by $S$, the corresponding $\widehat{S}$-realization is the
(geometric) complex $\vert L\vert_{\widehat{S}} = \{{\rm conv}(\{v_i\}_{i\in I}) \mid I\in L\}$, provided the vectors $\{v_i\}_{i\in I}$ are linearly independent for each $I\in L$.

\begin{example}{\rm (Standard geometric realization of a Bier sphere) In the case of a Bier sphere $\Bier(K) = B(K,K^\vee)\subset 2^{[m]\cup [m']}$ there is a standard   $\widehat{S}$-realization $\vert \Bier(K)\vert_s$, where
\begin{itemize}
  \item[(i)]  $S = [m]\cup [m']$, and
  \item[(ii)] $v_i = e_i$, for $i\in [m]$, and $v_{j'} = -e_j$, for $j'\in [m']$.
\end{itemize}
Note that $\vert \Bier(K)\vert_s \subset \diamondsuit^m$, where  $\diamondsuit^m := {\rm conv}\{\pm e_i\}_{i\in [m]}$ is the $m$-dimensional cross-polytope.
}
\end{example}

\medskip
Each simplex $\tau = {\rm conv}(\{v_i\}_{i\in I})$ has its own ambient cube $C_\tau$, defined as the Minkowski sum of line segments $C_\tau = \sum_{i\in I}[0,v_i]$.
Then the cubical complex $Front(C_\tau)$, via the corresponding radial map $i_c$, induces a cubulation of $\tau$, and all these cubulations piece together to define a cubulation of $\vert L\vert_{\widehat{S}}$. This cubulation is referred to as the $\widehat{S}$-frontal cubulation of $\vert L\vert_{\widehat{S}}$.

\begin{definition}{\rm (Canonical cubulation of Bier spheres)}\label{def:canonical_cubulation_Bier}
 Let $\Bier(K) = B(K,K^\vee)\subset 2^{[m]\cup [m']}$ be the Bier sphere associated to a simplicial complex $K\subset 2^{[m]}$ and its Alexander dual $K^\vee \subset 2^{[m']}$. Then its \emph{canonical cubulation} $B(K,K^\vee)_{cubic}$ is by definition the $\widehat{S}$-frontal cubulation where $S = [m]\cup [m']$ and $v_i := e_i$ for each $i\in [m]$ and
 $v_j := -e_j$ for each $j\in [m']$.
 \label{def:cubulation-Bier}
\end{definition}

The collection $\mathcal{C} := \{C_\tau \mid \tau\in L\}$ of cubes, described in the $\widehat{S}$-frontal cubulation of $\vert L\vert_{\widehat{S}}$, is itself a cubical complex. The following proposition easily follows from the observation (Figure \ref{fig:front-and-back}) that there exists a natural (radial) homeomorphism $I_c : Cone(\Delta^{m-1}) \rightarrow I^m$.

\begin{proposition}\label{prop:cubical-disc}
There is a topological isomorphism (homeomorphism) of topological spaces
\[
            Cone(\vert L\vert_{\widehat{S}}) \longrightarrow \vert \mathcal{C}\vert := \bigcup_{\tau\in L} C_\tau
\]
where $\vert L\vert_{\widehat{S}}$ is the $\widehat{S}$-realization of the simplicial complex $L$ and
$\vert \mathcal{C}\vert$ the union of all associated cubes $C_\tau$.
\end{proposition}

\section{Bier spheres meet polyhedral products}\label{sec:meet-1}

In this section we establish a natural and transparent link between Bier spheres and polyhedral products.

For convenience, a simplex $A\cup C'\in \Bier(K)$ is recorded as a triple $(A,B,C)$, where $B:= [m]\setminus (A\cup C)$.

\medskip  More explicitly, 
$(A,B,C)\in B(K,K^\vee)$ if and only if 
\begin{itemize}
  \item $[m] = A \uplus B \uplus C$ (disjoint union);
  \item  $A\in K$ and ${C'} := \{{i'}\}_{i\in C}\in K^\vee$;
  \item  $\emptyset \neq B \neq [m]$.
\end{itemize}

\subsection{Canonical cubulation of a Bier sphere and polyhedral products}

Let $I = [0,1]$,  $I_{\leqslant \frac{1}{2}} := [0,\frac{1}{2}], I_{\geqslant \frac{1}{2}} := [\frac{1}{2}, 1], I_{< \frac{1}{2}} := [0,\frac{1}{2})$, etc.
Similarly, for $J = [-1,1]$ let
$J_{\leqslant 0} := [-1,0], J_{\geqslant 0} := [0, 1], J_{< 0} := [-1,0)$, etc.

\medskip
The complex $\mathcal{Z}_K(I, I_{\leqslant \frac{1}{2}})\cap \mathcal{Z}_{K^\vee}(I, I_{\geqslant \frac{1}{2}})$  is clearly a cubical subcomplex of the $m$-cube $I^m$, subdivided into $2^m$ smaller cubes, and our objective is to clarify its structure. This complex is obviously isomorphic (as a cubical complex) to 
\[
\mathcal{Z}_K(J, J_{\leqslant 0})\cap \mathcal{Z}_{K^\vee}(J, J_{\geqslant 0}) =: Z(K,K^\vee)  \, .
\]

In the following proposition the cubical cells of the complex $Z(K,K^\vee)$ are labeled by triples
$(A,B,C)\in B(K,K^\vee)$, corresponding to simplices  in the extended face poset $B(K,K^\vee)^+$ of the Bier sphere $B(K,K^\vee)$.

\begin{proposition}\label{prop:decomp}
\begin{equation}\label{eqn:decomp}
    Z(K,K^\vee) = \bigcup_{(A,B,C)\in B(K,K^\vee)^+}  (J_{\geqslant 0 })^A \times \mbox{ $\{0\}^B$ } \times \mbox{ $(J_{\leqslant 0 })^C$ }
\end{equation}
where $B(K,K^\vee)^+ := B(K,K^\vee) \cup \{(\emptyset, [m], \emptyset)\}$.
\end{proposition}

\proof  Following Proposition \ref{prop:magic}, if $x \in Z(K,K^\vee)$ then
\begin{equation}\label{eqn:Bier-cubic-1}
 A:= x^{-1}(I_{> 0})\in K \quad \mbox{ {and} } C:= x^{-1}(I_{< 0})\in K^\vee
\end{equation}
determines a closed cubical cell on the right hand side of (\ref{eqn:decomp}), where $x$ belongs.

The converse is also true. If $(A,B,C)$ is a closed cubical cell of the smallest dimension, on the right hand side of (\ref{eqn:decomp}), where $x$ belongs, then the conditions (\ref{eqn:Bier-cubic-1}) are satisfied.
\qed

\medskip

In the following proposition we explain the connection of the complex $Z(K,K^\vee)$ and its boundary
$\partial Z(K,K^\vee)$ with the Bier sphere $B(K,K^\vee)$.

\begin{proposition}
 The cubical complex $Z(K,K^\vee)$ is a cubulation of an $(n-1)$-dimensional disc $D^{n-1}$ while its boundary
 $\partial Z(K,K^\vee)$ is the \emph{canonical cubulation}  of the Bier sphere $B(K,K^\vee)$. More explicitly,
\begin{itemize}
   \item The complex $Z(K,K^\vee)$ is precisely the cubical complex $\mathcal{C}$, described in Proposition \ref{prop:cubical-disc}, arising from the \emph{canonical cubulation} of the Bier sphere $B(K,K^\vee)$ ($\widehat{S}$-frontal cubulation) from Definition \ref{def:cubulation-Bier}.
  \item $\partial Z(K,K^\vee)$ is precisely the canonical, $\widehat{S}$-frontal cubulation $B(K,K^\vee)_{cubic}$ of the Bier sphere $B(K,K^\vee)$, described in Definition \ref{def:cubulation-Bier}.
\end{itemize}
\end{proposition}

\proof The proof is essentially by inspection, with the key observation that  $$(J_{\geqslant 0 })^A \times \mbox{ $\{0\}^B$ } \times \mbox{ $(J_{\leqslant 0 })^C$ }$$ is indeed the cube $C_\tau$ associated to the simplex $\tau = (A,B,C)$ in the Bier sphere $B(K,K^\vee)$.

The proof that $Z(K,K^\vee)$ is a disc follows from Proposition \ref{prop:cubical-disc} and the fact that $B(K,K^\vee)$ is a triangulation of a sphere.
\qed

\medskip

\begin{example}
Here is a simple example illustrating the relation
$\partial Z(K,K^\vee) {\cong} B(K,K^\vee)$.

\begin{figure}[htb]
    \centering
    \includegraphics[scale=0.5]{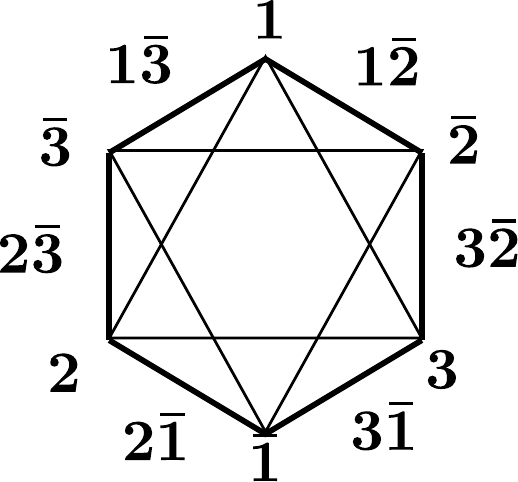}
    \caption{}
    \label{fig:prva}
\end{figure}
$K$ is the (self-dual) $0$-dimensional comp lex with three vertices.
\[
   \max(K) = V(K) = \{\{1\}, \{2\}, \{3\}\},   \quad  \max(K^\vee) = V(K^\vee) = \{\{1'\},\{2'\}, \{3'\}\}
\]
The associated Bier sphere is the hexagon (Figure \ref{fig:prva}),
\begin{equation}\label{eqn:example-1}
    \Bier(K,K^\vee) = \langle \{1,2'\}, \{3,2'\}, \{3,1'\}, \{2,1'\}, \{2,3'\}, \{1,3'\}\rangle \, .
\end{equation}

\noindent
The corresponding cubical complexes $Z(K,K^\vee)$  and $\partial Z(K,K^\vee)$ are depicted in Figure \ref{fig:druga}.

\begin{figure}[htb]
    \centering
    \includegraphics[scale=0.4]{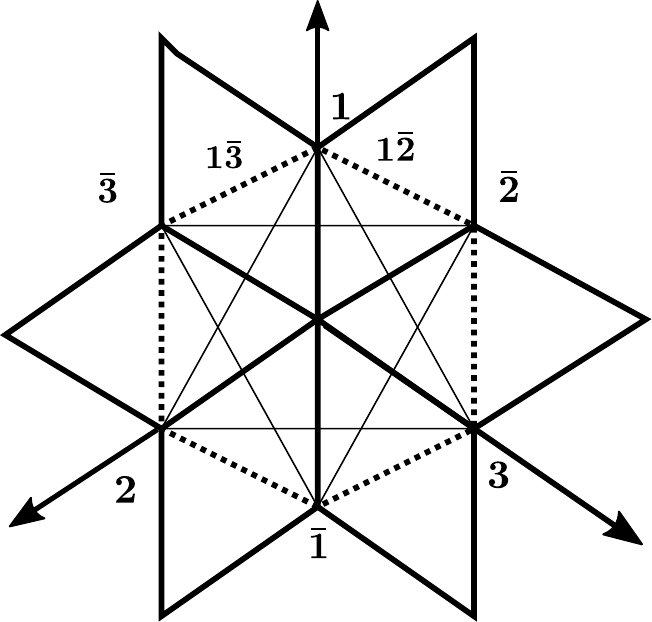}
    \caption{}
    \label{fig:druga}
\end{figure}
In light of the relation
\[
Z(K,K^\vee) = \bigcup_{(A,B,C)\in B(K,K^\vee)^+}  (J_{\geqslant 0})^A \times \mbox{ $\{0\}^B$ } \times \mbox{ $(I_{\leqslant 0})^C$ }
\]
and (\ref{eqn:example-1}), the complex $Z(K,K^\vee)$ is the union of six squares.
\end{example}

\subsection*{Acknowledgements}
The authors are grateful to Taras Panov, Matvey Sergeev, and Ale\v{s} Vavpeti\v{c} for various fruitful discussions, valuable comments and suggestions. 

The authors were supported by the Serbian Ministry of Science, Innovations and Technological Development through the Mathematical Institute of the Serbian Academy of Sciences and Arts.

R. \v Zivaljevi\' c was supported by the Science Fund of the Republic of Serbia, Grant No.\ 7744592, Integrability and Extremal Problems in Mechanics, Geometry and Combinatorics - MEGIC.

\normalsize

\end{document}